\documentclass[11pt]{article}
\usepackage{amsmath,amssymb,amsfonts}
\usepackage[dvips]{graphicx}

\newenvironment{proof}{\noindent {\em {Proof}}:}{$\square$
\medskip}
\newtheorem{theorem}{Theorem}[section]
\newtheorem{corollary}[theorem]{Corollary}
\newtheorem{lemma}[theorem]{Lemma}
\newtheorem{remark}[theorem]{Remark}
\newtheorem{proposition}[theorem]{Proposition}
\newtheorem{definition}[theorem]{Definition}
\newtheorem{example}[theorem]{Example}
\newtheorem{question}[theorem]{Question}
\newtheorem{problem}[theorem]{Problem}

\newcommand{\bd}[1]{\begin{definition}\rm\label{#1}}
\newcommand{\bt}[1]{\begin{theorem}\label{#1}}
\newcommand{\bc}[1]{\begin{corollary}\label{#1}}
\newcommand{\bl}[1]{\begin{lemma}\label{#1}}
\newcommand{\bp}[1]{\begin{proposition}\label{#1}}
\newcommand{\be}[1]{\begin{example}\rm\label{#1}}
\newcommand{\bq}[1]{\begin{question}\rm\label{#1}}
\newcommand{\bprob}[1]{\begin{problem}\rm\label{#1}}
\newcommand{\beq}[1]{\begin{eqnarray}\label{#1}}
\newcommand{\br}[1]{\begin{remark}\rm\label{#1}}

\newcommand{\el}{\end{lemma}}
\newcommand{\ep}{\end{proposition}}
\newcommand{\ee}{\end{example}}
\newcommand{\eq}{\end{question}}
\newcommand{\eprob}{\end{problem}}
\newcommand{\eeq}{\end{eqnarray}}
\newcommand{\ed}{\end{definition}}
\newcommand{\et}{\end{theorem}}
\newcommand{\ec}{\end{corollary}}
\newcommand{\er}{\end{remark}}
\begin{document}
\title{Discrete Galerkin Method for Fractional Integro-Differential Equations}
\author{ P. Mokhtary \\
\small{Department of Mathematics, Faculty of basic Sciences}\\
\small{Sahand University of Technology, Tabriz, Iran}\\
\footnotesize{E-Mails: mokhtary.payam@gmail.com \;\;
mokhtary@sut.ac.ir}} \maketitle


\begin{abstract}
In this paper, we develop a fully discrete Galerkin method for
solving initial value fractional integro-differential
equations(FIDEs). We consider Generalized Jacobi polynomials(GJPs)
with indexes corresponding to the number of homogeneous initial
conditions as natural basis functions for the approximate solution.
The fractional derivatives are used in the Caputo sense. The
numerical solvability of algebraic system obtained from
implementation of proposed method for a special case of FIDEs is
investigated. We also provide a suitable convergence analysis to
approximate solutions under a more general regularity assumption on
the exact solution.
\end{abstract}



\maketitle

\vspace{-6pt}

\section{Introduction}
In this paper, we provide a convergent numerical scheme for solving
FIDE
\begin{equation}\label{1}
\left\{\begin{array}{l}
\mathcal D^q u(x)=p(x) u(x)+f(x)+\lambda \int\limits_0^x{K(x,t) u(t) dt},~~~ x \in \Omega=[0,1],\\
\\
u(0)=0,
\end{array}\right.
\end{equation}
where $q\in \mathbb R^+ \bigcap (0,1)$. The symbol $\mathbb R^+$ is
the collection of all positive real numbers.~$p(x)$ and $f(x)$ are
given continuous functions and $K(x,t)$ is a given sufficiently
smooth kernel function, $u(x)$ is the unknown function.

Note that the condition $u(0)=0$ is not restrictive, due to the fact
that (\ref{1}) with nonhomogeneous initial condition $u(0)=d,~~d
\neq 0$ can be converted to the following homogeneous FIDE
\begin{equation*}
\left\{\begin{array}{l}
\mathcal D^q \tilde u(x)=p(x) \tilde u(x)+\tilde f(x)+\lambda \int\limits_0^x{K(x,t) \tilde u(t) dt},~~~ x \in \Omega=[0,1],\\
\\
\tilde u(0)=0,
\end{array}\right.
\end{equation*}
by the simple transformation $\tilde u(x)=u(x)-d$, where $\tilde
f(x)=f(x)+d\bigg(p(x)+\lambda \int_{0}^{x}{K(x,t)dt}\bigg)$.

Such kind of equations arising in the mathematical modeling of
various physical phenomena, such as heat conduction, materials with
memory, combined conduction, convection and radiation
problems(\cite{r2}, \cite{r5}, \cite{r20}, \cite{r21}).

$\mathcal D^q u(x)$ denotes the fractional Caputo differential
operator of order $q$ and defines as(\cite{r8}, \cite{r13},
\cite{r22})
\begin{equation}
\label{2} \mathcal D^q u(x) = \mathcal I^{1-q} u'(x),
\end{equation}
where
\begin{equation}\label{3}
\mathcal I^\mu u(x)=\frac{1}{\Gamma{(\mu)}}
\int\limits_0^x{(x-s)^{\mu-1} u(s) ds},
\end{equation}
is the fractional integral operator from order $\mu$.
$\Gamma{(\mu)}$ is the well known Gamma function. The following
relation holds\cite{r8}
\begin{equation}\label{20}\mathcal I^q(\mathcal D^q
u(x))=u(x)-u(0).\end{equation}

From the relation above, it is easy to check that (\ref{1}) is
equivalent with the following weakly singular Volterra integral
equation
\begin{equation}\label{5}
u(x)=g(x)+\lambda \int\limits_0^x{{\bar K}(x,t) u(t) dt}.
\end{equation}

Here $g(x)=\mathcal I^q f(x)$ and ${\bar
K}(x,t)=\frac{(x-t)^{q-1}}{\Gamma{(q)}}p(t)+\int\limits_t^x{\frac{(x-s)^{q-1}}{\Gamma{(q)}}
K(s,t)ds}.$ From the well known existence and uniqueness
Theorems(\cite{r3}, \cite{r7}), it can be concluded that if the
following conditions are fulfilled
\begin{itemize}
    \item $f(x) \in C^l(\Omega),~~l \ge 1$
    \item $p(x) \in C^l(\Omega),~~l \ge 1$
    \item $K(x,t) \in C^l(D),~~ D=\{(x,t);0 \le t \le x\le
1\},~~l \ge1$
     \item $K(x,x)\neq 0$,
\end{itemize}
the regularity of the unique solution $u(x)$ of (\ref{5}) and also
(\ref{1}) is described by
\begin{equation}\label{6}
u(x)=\sum\limits_{(j,k)}{\gamma_{j,k} x^{j+kq}}+U_l(x;q) \in
C^l(0,1]\bigcap C(\Omega),\hspace{.5 cm} \text{with} \hspace{.5 cm}
|u'(x)| \le C_q x^{q-1},
\end{equation}
where the coefficients $\gamma_{j,k}$ are some constants, $U_l(.;q)
\in C^l(\Omega)$ and $(j,k):=\{(j,k):~~j,k \in \mathbb
N_0,~j+kq<l\}$. Here $\mathbb N_0=\mathbb N \bigcup \{0\}$, where
the symbol $\mathbb N$ denotes the collection of all natural
numbers. Thus, we must expect the first derivative of the solution
to has a discontinuity at the origin. More precisely, if the given
functions $g(x), p(x)$ and $K(x,t)$ are real analytic in their
domains then it can be concluded that there is a function
$U=U(z_1,z_2)$ real and analytic at $(0,0)$, so that solutions of
(\ref{5})and also (\ref{1}) can be written as
$u(x)=U(x,x^q)$(\cite{r3}, \cite{r7}).

Recently, several numerical methods for the numerical solution of
FIDE's have been proposed. In \cite{r19}, fractional differential
transform method was developed to solve FIDE's with nonlocal
boundary conditions. In \cite{r23}, Rawashdeh studied the numerical
solution of FIDE's by polynomial spline functions. In \cite{r1}, an
analytical solution for a class of FIDE's was proposed. Adomian
decomposition method to solve nonlinear FIDE's was proposed in
\cite{r17}. In \cite{r25}, authors solved fractional nonlinear
Volterra integro differential equations using the second kind
Chebyshev wavelets. In \cite{r11}, Taylor expansion approach was
presented for solving a class of linear FIDE's including those of
Fredholm and Volterra types. In \cite{r16}, authors were solved
FIDE's by adopting Hybrid Collocation method to an equivalent
integral equation of convolution type. In \cite{r12}, Chebyshev
Pseudospectral method was implemented to solve linear and nonlinear
system of FIDE's. In \cite{r15}, authors proposed an analyzed
spectral Jacobi Collocation method for the numerical solution of
general linear FIDE's. In \cite{r9}, authors applied Collocation
method to solve the nonlinear FIDE's. In \cite{r18}, Mokhtary and
Ghoreishi, proved the $L^2$ convergence of Legendre Tau method for
the numerical solution of nonlinear FIDE's.

Many of the techniques mentioned above or have not proper
convergence analysis or if any, very restrictive conditions
including smoothness of the exact solution are assumed. In this
paper we will consider non smooth solutions of (\ref{1}). In this
case although the discrete Galerkin method can be implemented
directly but this method leads to very poor numerical results. Thus
it is necessary to introduce a regularization procedure that allows
us to improve the smoothness of the given functions and then to
approximate the solution with a satisfactory order of convergence.
To this end, we propose a regularization process which the original
equation (\ref{1}) will be changed into a new equation which
possesses a more regularity properties by taking a suitable
coordinate transformation. Our logic in choosing proper
transformation is based upon the formal asymptotic expansion of the
exact solution in (\ref{6}). Consider (\ref{1}), using the variable
transformation
\begin{equation}\label{6xx}
x=v^{\frac{1}{q}},\;\; v=x^{q},\;\; t=w^{\frac{1}{q}},\;\; w=t^q,
\end{equation}
we can change (\ref{1}) to the following equation
\begin{equation}\label{6x}
\mathcal M^q \bar u(v)=\bar p(v) \bar u(v)+\bar
f(v)+\lambda\int\limits_0^v{\tilde{K}(v,w) \bar{u}(w)dw},
\end{equation}
where
\begin{eqnarray}\label{rv4}
\nonumber\bar p(v)&=&p(v^{\frac{1}{q}}),\;\; \bar
f(v)=f(v^{\frac{1}{q}}),~~{\tilde
K}(v,w)=\frac{w^{{\frac{1}{q}}-1}}{q}
K(v^{\frac{1}{q}},w^{\frac{1}{q}}).\\
\mathcal M^q \bar
u(v)&=&\frac{1}{\Gamma{(1-q)}}\int\limits_0^v{(v^{\frac{1}{q}}-w^{\frac{1}{q}})^{-q}
{\bar u}'(w)dw}.
\end{eqnarray}

From (\ref{6}), the exact solution $\bar u(v)$ can be written as
$\bar{u}(v)=u(v^{\frac{1}{q}})=\sum\limits_{(j,k)}{\gamma_{j,k}
v^{\frac{j}{q}+k}}+U_l(v^{\frac{1}{q}};q)$. It can be easily seen
that $\bar u'(v) \in C(\Omega)$. It is trivial that for
$q=\frac{1}{n},~n \in \mathbb N$, the unknown function $\bar u(v)$
will be in the form
\[
\bar{u}(v)=u(v^n)=\sum\limits_{(j,k)}{\gamma_{j,k}
v^{nj+k}}+U_l(v^n;q), \quad n \in \mathbb N,
\]
which is infinitely smooth. Then we can deduce that the solution
$\bar u(v)$ of the new equation (\ref{6x}) possesses better
regularity and discrete Galerkin theory can be applied conveniently
to obtain high order accuracy.

In the sequel, we introduce the discrete Galerkin solution $\bar
u_N(v)$ based upon GJPs to (\ref{6x}). Since the exact solutions of
(\ref{1}) can be written as $u(x)=\bar u(v)$ then we define
$u_N(x)=\bar u_N(v),\; x, v \in \Omega$ as the approximate solution
of (\ref{1}).

Spectral Galerkin method is one of the weighted residual
methods(WRM), in which approximations are defined in terms of
truncated series expansions, such that residual which should be
exactly equal to zero, is forced to be zero only in an approximate
sense. It is well known that, in this method, the expansion
functions must satisfy in the supplementary conditions. The two main
characteristics behind the approach are that, first it reduces the
given problems to those of solving a system of algebraic equations,
and in general converges exponentially and almost always supplies
the most terse representation of a smooth solution(\cite{a13},
\cite{a14}, \cite{aa26}).

In this article, we use shifted GJPs on $\Omega$, which are mutually
orthogonal with respect to the shifted weight function
$\delta^{\alpha,\beta}(v)=(2-2v)^\alpha(2v)^\beta$ on $\Omega$ where
$\alpha, \beta$ belong to one of the following index sets
\begin{eqnarray*}
{\mathcal N_1}&=&\{(\alpha,\beta); \alpha, \beta \le -1, ~\alpha,
\beta \in \mathbb Z\},\quad~~~~~~~~~~~~~~ {\mathcal
N_2}=\{(\alpha,\beta); \alpha \le -1, \beta > -1,~~\alpha \in
\mathbb Z, \beta \in \mathbb R\},
 \\
{\mathcal N_3}&=&\{(\alpha,\beta); \alpha>-1, \beta \le -1,~\alpha
\in \mathbb R, \beta \in \mathbb Z\},\quad {\mathcal
N_4}=\{(\alpha,\beta); \alpha, \beta
> -1,~\alpha, \beta \in \mathbb R\},
\end{eqnarray*}
where the symbol $\mathbb Z$ is the collection of all integer
numbers. The main advantage of GJPs is that these polynomials, with
indexes corresponding to the number of homogeneous initial
conditions in a given FIDE, are the natural basis functions to the
Galerkin approximation of this problem(\cite{a15}, \cite{a16}).

The organization of this paper is as follows: we begin by reviewing
some preliminaries which are required for establishing our results
in Section 2. In Section 3, we introduce the discrete Galerkin
method based on the GJPs and its application to (\ref{6x}).
Numerical solvability of the algebraic system obtained from discrete
Galerkin discretization of a special case of (\ref{6x}) with
$0<q<\frac{1}{2}$ and $\bar p(v)=1$ based on GJPs is given in
Section 4. Convergence analysis of the proposed scheme is provided
in Section 5. Numerical experiments are carried out in Section 6.
\section{Preliminaries and Notations}
In this section, we review the basic definitions and properties that
are required in the sequel.

Defining weighted inner product
\[ \Big(u_1,u_2\Big)_{\alpha, \beta}=\int_{\Omega}{u_1(v) u_2(v) \delta^{\alpha, \beta}(v)
dv},\] and discrete Jacobi-Gauss inner product
\[\bigg(u_1,u_2\bigg)_{N,\alpha,\beta}=\sum\limits_{k=0}^N{u_1(v_k^{\alpha,\beta}) u_2(v_k^{\alpha,\beta}) \delta_k^{\alpha,\beta}},\]
we recall the following norms over $\Omega$
\[
\|u\|_{\alpha,\beta}^2=\Big(u,u\Big)_{\alpha,\beta}, \quad
\|u\|_{N,\alpha,\beta}^2=\Big(u,u\Big)_{N,\alpha,\beta},\quad
\|u\|_{\infty}=\sup_{v \in \Omega} |u(v)|.\]

Here, $v_k^{\alpha,\beta}$ and $\delta_k^{\alpha,\beta}$ are the
shifted Jacobi Gauss quadrature nodal points on $\Omega$ and
corresponding weights respectively.

The non-uniformly Jacobi-weighted Sobolev space denotes by
$B_{\alpha,\beta}^{k}(\Omega)$ and defines as follows
\[
B_{\alpha , \beta}^{k}(\Omega)=\{ u: \|u^{(s)}\|_{\alpha+s,\beta+s}
< \infty;~~ 0 \le s\le k\},\] equipped with the norm and semi-norm
\[
||u||_{\alpha,\beta,k}^2=\sum\limits_{s = 0}^k
||u^{(s)}||_{\alpha+s,\beta+s}^2, \quad
|u|_{\alpha,\beta,k}=||u^{(k)}||_{\alpha+k,\beta+k}.
\]

The space $B_{\alpha,\beta}^{k}(\Omega)$ distinguishes itself from
the usual weighted Sobolev space $H_{\alpha,\beta}^{k}(\Omega)$ by
involving different weight functions for derivatives of different
orders. The usual weighted Sobolev space
$H_{\alpha,\beta}^{k}(\Omega)$ is defined as
\[
H_{\alpha , \beta}^{k}(\Omega)=\{ u: \|u^{(s)}\|_{\alpha,\beta} <
\infty;~~ 0 \le s\le k\},\] equipped with the norm
\[
||u||_{H_{\alpha,\beta}^{k}(\Omega)}^2=\sum\limits_{s = 0}^k
||u^{(s)}||_{\alpha,\beta}^2.
\]

We denote the shifted GJPs on $\Omega$ by $G_n^{\alpha,\beta}(v)$
and define as
\begin{equation}
\label{7} G_n^{\alpha,\beta}(v) = \left\{
\begin{array}{l}
 {(2 -2v)^{-\alpha}}{(2v)^{-\beta}}J_{n-n_0}^{-\alpha,-\beta}(v),~ (\alpha,\beta) \in {\mathcal N_1}, ~~n_0 = -(\alpha+\beta),\\
 \\
 {(2 -2v)^{-\alpha}}J_{n-n_0}^{-\alpha,\beta}(v), \quad \quad \quad ~~~\;\;(\alpha,\beta) \in {\mathcal N_2},~~ n_0 = -\alpha, \\
 \\
 {(2v)^{-\beta}}J_{n-n_0}^{\alpha,-\beta}(v), \quad \quad \quad ~~~\;\;(\alpha,\beta) \in {\mathcal N_3},~~ n_0 = -\beta,\\
 \\
 J_{n-n_0}^{\alpha,\beta}(v),\quad  \quad \quad \;\;\quad \quad \quad~~~~~~~ \;(\alpha,\beta) \in {\mathcal N_4},~~ n_0=0,\\
 \end{array} \right.
 \end{equation}
where $J_n^{\alpha,\beta}(v)$ is the classical shifted Jacobi
polynomials on $\Omega$; see \cite{aa26}. An important fact is that
the shifted GJPs $\{G_n^{\alpha,\beta}(v); n \ge 1 \}$ form a
complete orthogonal system in $L_{\alpha,\beta}^2(\Omega)$;
see(\cite{a15}, \cite{a16}). To present a Galerkin solution for
(\ref{6x}) it is fundamental that the basis functions in the
approximate solution satisfy in the homogeneous initial condition.
To this end, since $G_n^{0,-1}(0) = 0,\;\; n \ge 1$, then we can
consider $\{G_n^{0,-1}(v),~~n\geq 1\}$ as suitable basis functions
to the Galerkin solution of (\ref{6x}).

From (\ref{7}) and the following formula \cite{aa16}
\[
J_i^{\alpha,\beta}(v)=\sum\limits_{k=0}^{i}{(-1)^{i-k}\frac{\Gamma{(i+\beta+1)}\Gamma{(i+k+\alpha+\beta+1)}}
{\Gamma{(k+\beta+1)}\Gamma{(i+\alpha+\beta+1)}(i-k)!k!}v^k},\quad
\alpha,\beta\in \mathcal N_4,
\]
we can obtain the following explicit formula for $G_i^{0,-1}(v)$
\begin{equation}\label{8} G_i^{0,-1}(v)= (2v)J_{i-1}^{0,1}(v)=2
\sum_{k=0}^{i-1}(-1)^{i-1-k} \frac{(i+k)!}{(k+1)! (i-1-k)!
k!}v^{k+1},~~~~~i \ge 1.\end{equation}

For any continuous function $Z(v)$ on $\Omega$, we define the
Legendre Gauss interpolation operator $\mathcal I_N$, as
\begin{equation}\label{9}
\mathcal I_N Z(v) =\sum\limits_{s=0}^N
{\frac{\bigg(Z,J_s^{0,0}\bigg)_{N,0,0}}{\|J_s^{0,0}\|_{N,0,0}^2}
J_s^{0,0}(v)}.\end{equation}

Let $\mathcal P_N$ be the space of all algebraic polynomials of
degree up to $N$. We introduce Legendre projection $\Pi_N:
L^2(\Omega) \to \mathcal P_N$ which is a mapping such that for any
$Z(v) \in L^2(\Omega)$,
\begin{equation}\label{cc5}
\bigg(Z-\Pi_N Z, \phi\bigg)_{0,0}=0,\quad \forall \phi \in \mathcal
P_N.
\end{equation}
\section{Discrete Galerkin Approach}
In this section, we present the numerical solution of (\ref{6x}) by
using the discrete Galerkin method based on GJPs.

Let
\begin{equation}\label{10c}
\tilde u_N(v)=\sum\limits_{i=1}^N{b_i G_{i}^{0,-1}(v)},
\end{equation}
be the Galerkin solution of (\ref{6x}). It is trivial that $\tilde
u_N(0)=0.$

Galerkin formulation of (\ref{6x}) is to find $\tilde u_N(v)$, such
that
\begin{equation}\label{rv17}
\bigg(\mathcal M^q \tilde u_N,G_i^{0,-1}\bigg)_{0,-1}=\bigg(\bar
p(v) \tilde u_N(v),G_i^{0,-1}\bigg)_{0,-1}+\bigg(\bar
f(v),G_i^{0,-1}\bigg)_{0,-1}+\lambda\bigg(\mathcal K(\tilde
u_N),G_i^{0,-1}\bigg)_{0,-1}, ~~i=1,2,...,N
\end{equation}
where $\mathcal K(\tilde u_N)=\int\limits_0^v{\tilde K(v,w)\tilde
u_N(w)dw}.$

Applying transformation $w(\theta)=v \theta,~~\theta \in \Omega$~ we
get
\begin{equation}\label{rv5}
\mathcal K(\tilde u_N)=\mathcal K_\theta(\tilde u_N)=v
\int\limits_0^1{\tilde K(v,w(\theta)) \tilde u_N(w(\theta))d\theta}.
\end{equation}

Substituting (\ref{rv5}) in (\ref{rv17}) yields
\begin{multline}\label{12}
\bigg(\mathcal M^q \tilde u_N,G_i^{0,-1}\bigg)_{0,-1}=\bigg(\bar
p(v) \tilde u_N(v),G_i^{0,-1}\bigg)_{0,-1}+\bigg(\bar
f(v),G_i^{0,-1}\bigg)_{0,-1}+\lambda\bigg(\mathcal K_\theta(\tilde
u_N),G_i^{0,-1}\bigg)_{0,-1},\\ i=1,2,...,N.
\end{multline}

Inserting (\ref{10c}) in (\ref{12}) we get
\begin{multline}\label{13}
\sum\limits_{j=1}^{N}{b_j \bigg\{\bigg(\mathcal M^q
G_j^{0,-1}(v),G_i^{0,-1}\bigg)_{0,-1}-\bigg(\bar p(v)
G_j^{0,-1}(v),G_i^{0,-1}\bigg)_{0,-1}-\lambda\bigg(\mathcal
K_{\theta}(G_j^{0,-1}),G_i^{0,-1}\bigg)_{0,-1}\bigg\}}
\\ =\bigg(\bar f(v),G_i^{0,-1}\bigg)_{0,-1},
~~i=1,2,...,N.\hspace{.01 cm}
\end{multline}

Following the relation $G_i^{0,-1}(v)
\delta^{0,-1}(v)=G_{i-1}^{0,1}(v)$, we can rewrite (\ref{13}) as
\begin{multline}\label{14}
\sum\limits_{j=1}^{N}{b_j \bigg\{\bigg(\mathcal M^q
G_j^{0,-1}(v),G_{i-1}^{0,1}\bigg)_{0,0}-\bigg(\bar p(v)
G_j^{0,-1}(v),G_{i-1}^{0,1}\bigg)_{0,0}-\lambda\bigg(\mathcal
K_{\theta}(G_j^{0,-1}),G_{i-1}^{0,1}\bigg)_{0,0}\bigg\}}
\\=\bigg(\bar
f(v),G_{i-1}^{0,1}\bigg)_{0,0},~~ i=1,2,...,N.
\end{multline}

Now, we try to find an explicit form for $\mathcal M^q G_j^{0,-1}$.
To this end, using (\ref{8}) we have
\begin{eqnarray}\label{15}
\mathcal M^q G_j^{0,-1}(v)&=&
\frac{2}{\Gamma{(1-q)}}\sum\limits_{k=0}^{j-1}{(-1)^{j-1-k}\frac{(j+k)!}{(k+1)!k!
(j-1-k)!}
\int\limits_0^v{\big(v^{\frac{1}{q}}-w^{\frac{1}{q}}\big)^{-q}(w^{k+1})'dw}}\\
\nonumber
&=&\frac{2}{\Gamma{(1-q)}}\sum\limits_{k=0}^{j-1}{(-1)^{j-1-k}\frac{(j+k)!}{(k!)^2
(j-1-k)!}
\int\limits_0^v{\big(v^{\frac{1}{q}}-w^{\frac{1}{q}}\big)^{-q} w^{k}
dw}}.
\end{eqnarray}

Applying the relation\cite{ax15}
\begin{equation*}
\int\limits_0^v{\bigg(v^{\frac{1}{q}}-w^{\frac{1}{q}}\bigg)^{-q} w^k
dw}=\bigg(\frac{q \pi \csc{(\pi q)}\Gamma{(q+qk)}}{\Gamma{(q)}
\Gamma{(1+k q)}}\bigg)v^k,~~ k \ge 0,
\end{equation*}
in (\ref{15}) we can obtain the following explicit formula for
$\mathcal M^q G_j^{0,-1}$:
\begin{equation}\label{15x}
\mathcal M^q G_j^{0,-1}(v)
=\frac{2}{\Gamma{(1-q)}}\sum\limits_{k=0}^{j-1}{(-1)^{j-1-k}\frac{(j+k)!}{(k!)^2
(j-1-k)!}\bigg(\frac{q \pi \csc{(\pi q)}\Gamma{(q+qk)}}{\Gamma{(q)}
\Gamma{(1+k q)}}\bigg)v^k}=:\Psi_{j,q}(v),
\end{equation}

Substituting (\ref{15x}) in (\ref{14}) we obtain
\begin{multline}\label{16}
\sum\limits_{j=1}^{N}{b_j
\bigg\{\bigg(\Psi_{j,q}(v),G_{i-1}^{0,1}\bigg)_{0,0}-\bigg(\bar p(v)
G_j^{0,-1}(v),G_{i-1}^{0,1}\bigg)_{0,0}-\lambda\bigg(\mathcal
K_{\theta}(G_j^{0,-1}),G_{i-1}^{0,1}\bigg)_{0,0}\bigg\}}\\=\bigg(\bar
f(v),G_{i-1}^{0,1}\bigg)_{0,0},~~i=1,2,...,N.
\end{multline}

In this position, we approximate the integral terms of (\ref{16})
using $(N+1)-$point Legendre Gauss quadrature formula. Our discrete
Galerkin method is to seek \begin{equation}\label{10}\bar
u_N(v)=\sum\limits_{i=1}^N{a_i G_{i}^{0,-1}(v)},\end{equation} such
that coefficients $\{a_j\}_{j=1}^N$ satisfies in the following
algebraic system of linear equations
\begin{multline}\label{17}
\sum\limits_{j=1}^{N}{a_j
\bigg\{\bigg(\Psi_{j,q}(v),G_{i-1}^{0,1}\bigg)_{0,0}-\bigg(\bar p(v)
G_j^{0,-1}(v),G_{i-1}^{0,1}\bigg)_{N,0,0}-\lambda\bigg(\mathcal
K_{N,\theta}(G_j^{0,-1}),G_{i-1}^{0,1}\bigg)_{N,0,0}\bigg\}}
\\=\bigg(\bar f(v),G_{i-1}^{0,1}\bigg)_{N,0,0}, i=1,2,...,N,
\end{multline}
where
\begin{equation}\label{rv6}
\mathcal K_{N,\theta}(G_j^{0,-1})=v \sum\limits_{k=0}^{N}{\tilde
K(v,w(\theta_k)) G_j^{0,-1}(w(\theta_k))\delta_k}.
\end{equation}

Here $\{\theta_k\}_{k=0}^N$ and $\{\delta_k\}_{k=0}^N$ are the
shifted Legendre Gauss quadrature points on $\Omega$ and
corresponding weights respectively. Note that, from (\ref{15x}) we
can see that $\Psi_{j,q}(v)$ is a polynomial from degree at most
$N$, then we have
$\bigg(\Psi_{j,q}(v),G_{i-1}^{0,1}\bigg)_{N,0,0}=\bigg(\Psi_{j,q}(v),G_{i-1}^{0,1}\bigg)_{0,0}.$
It is trivial that the solution of (\ref{17}) gives us unknown
coefficients $\{a_i\}_{i=1}^N$ in (\ref{10}).
\section{Existence and Uniqueness Theorem for Discrete Galerkin Algebraic System}
The main object of this section is providing an existence and
uniqueness Theorem for a special case of the discrete Galerkin
algebraic system of equations (\ref{17}) with $\bar p(v)=1$ and
$0<q<\frac{1}{2}$. Throughout the paper, $C_i$ will denote a generic
positive constant that is independent on $N$.

First, we give some preliminaries which will be used in the sequel.
\begin{definition}
Let $\mathcal X, \mathcal Y$ be normed spaces. A linear operator
$\mathcal A: \mathcal X \to \mathcal Y $ is compact if the set
$\{\mathcal A {x}|~ ||x||_{\mathcal X}\leq 1\}$ has compact closure
in $\mathcal Y$. 
\end{definition}
\begin{theorem}\label{rt1}[\ref{rvv1}, \ref{rvv2},
\ref{rvv3}] Assume $\mathcal X,~\mathcal Y$ be two banach spaces.
Let
\begin{equation}\label{rv7}v=\mathcal A v+f,
\end{equation}
be a linear operator equation where $\mathcal A:\mathcal X \to
\mathcal Y$ is a linear continuous operator, and the operator
$I-\mathcal A$ is continuously invertible. As an approximation
solution of (\ref{rv7}) we consider the equation
\begin{equation}\label{rv8c}
v_N=\mathcal A_N v_N+\mathcal B_N f,
\end{equation}
which can be rewritten as
\begin{equation}\label{rv8}
v_N=\tilde{\mathcal B}_N \mathcal A v_N+\mathcal S_N v_N+\mathcal
B_N f,
\end{equation}
where $\mathcal A_N$ is a linear continuous operator in a closed
subspace $\tilde{\mathcal Y}$ of $\mathcal Y$. $\mathcal B_N,
\tilde{\mathcal B}_N:\mathcal Y\to \tilde{\mathcal Y}$ are linear
continuous projection operators and $\mathcal S_N=\mathcal
A_N-\mathcal{\tilde B}_N \mathcal A$ is a linear operator in
$\tilde{\mathcal Y}$. If the following conditions are fulfilled
\begin{itemize}
    \item  for any $Z \in \tilde{\mathcal Y}$ we have $\|\mathcal S_N Z\|\to 0$ ~as~$N \to \infty$
    \item $\|\mathcal A-\tilde{\mathcal B}_N \mathcal A\| \to 0$~as ~$N \to \infty$
    \item $\|f-\mathcal B_N f\| \to 0$~ as ~ $N \to \infty$
\end{itemize}
then (\ref{rv8}) possesses a uniquely solution $v_N \in
\tilde{\mathcal Y}$, for a sufficiently large $N$.
\end{theorem}
\begin{lemma}\label{rl1}
\cite{rvv4} Let $\mathcal X, \mathcal Y$ be banach spaces and
$\tilde{\mathcal Y}$ be a subspace of $\mathcal Y$. Let
$\tilde{\mathcal B}_N:\mathcal Y \to \tilde{\mathcal Y}$ be a family
of linear continuous projection operators with
\[
\tilde{\mathcal B}_N y \to y~~ \text{as}~~ N \to \infty, ~~ y \in
\mathcal Y.
\]

Assume that linear operator $\mathcal A:\mathcal X \to \mathcal Y$
be compact. Then
\[
\|\mathcal A-\tilde{\mathcal B}_N \mathcal A\| \to 0~~\text{as}~~ N
\to \infty.
\]
\end{lemma}
\begin{lemma}\label{l1} (Interpolation error bound\cite{aa26}) Let
$\mathcal I_N Z$ be the interpolation polynomial approximation of
the function $Z(v)$ defined in (\ref{9}). For any $Z(v) \in
B_{0,0}^k(\Omega)$ with $k \ge 1$, we have
\[
\|Z-\mathcal I_N Z\|_{0,0} \le C N^{-k} |Z|_{0,0,k}.
\]
\end{lemma}
\begin{lemma}\label{l3} \cite{r6}
For every bounded function $Z(v)$, there exists a constant $C$
independent of $Z$ such that
\[
\sup\limits_{N}\|\mathcal I_N Z\|_{0,0} \le C \sup\limits_{v}|Z(v)|.
\]
\end{lemma}

\begin{lemma}\label{l4} (Legendre Gauss quadrature error bound\cite{aa26})
If $Z(v) \in B_{0,0}^k(\Omega)$ for some $k \ge 1$ and $\Phi \in
\mathcal P_N$, then for the Legendre Gauss integration we have

\[
\bigg|(Z,\Phi)_{0,0}-(Z,\Phi)_{N,0,0}\bigg| \le C N^{-k}
\|Z\|_{0,0,k} \|\Phi\|_{0,0}.
\]
\end{lemma}

Now we intend to prove existence and uniqueness Theorem for a
special case of the discrete Galerkin system (\ref{17}) with $\bar
p(v)=1$ and $0<q<\frac{1}{2}$.
\begin{theorem}
(Existence and Uniqueness)Let $0<q<\frac{1}{2}$ and $\bar p(v)=1$.
If (\ref{6x}) has a uniquely solution $\bar u(v)$ then the linear
discrete Galerkin system (\ref{17}) has a uniquely solution $\bar
u_N(v) \in \mathcal P_N$ for sufficiently large $N$.
\end{theorem}
\begin{proof}
Our strategy in proof is based on two steps. First, we try to
represent (\ref{17}) in the operator form (\ref{rv8}). Then by
applying Theorem \ref{rt1} to operator form obtained in the first
step the desired result have been concluded.

{\bf Step 1:} In this step, we show that the discrete Galerkin
system (\ref{17}) can be written in the operator form (\ref{rv8}).
To this end, consider (\ref{6x}) and define
\[
\mathcal{\bar R}_N(v)=\mathcal M^q \bar u_N(v)-\bar u_N(v)-\lambda
\mathcal K_{N,\theta}(\bar u_N)-\bar f(v).
\]

According to the proposed method, we have
\[
\bigg(\mathcal{\bar R}_N(v),G_{i-1}^{0,1}(v)\bigg)_{N,0,0}=0,\quad
i=1, 2,...,N.
\]

From interpolation and Legendre Gauss quadrature properties, we can
write
\begin{equation}\label{18}
\bigg(\mathcal{\bar
R}_N(v),G_{i-1}^{0,1}(v)\bigg)_{N,0,0}=\bigg(\mathcal
I_N(\mathcal{\bar
R}_N),G_{i-1}^{0,1}(v)\bigg)_{N,0,0}=\bigg(\mathcal
I_N(\mathcal{\bar R}_N),G_{i-1}^{0,1}(v)\bigg)_{0,0}=0,\quad i=1,
2,...,N.
\end{equation}

Since $\mathcal I_N(\mathcal{\bar R}_N(v))$ is a polynomial, it can
be represented by a linear orthogonal polynomial expansion based on
$\big\{G_i^{0,-1}(v)\big\}_{i=0}^N$, as
\begin{equation}\label{cc6}
\mathcal I_N(\mathcal{\bar
R}_N)=\sum\limits_{i=1}^{N}{\frac{\bigg(\mathcal I_N(\mathcal{\bar
R}_N),G_{i}^{0,-1}\bigg)_{0,-1}}{\|G_i^{0,-1}\|_{0,-1}^2}
G_i^{0,-1}(v)}=\sum\limits_{i=1}^{N}{\frac{\bigg(\mathcal
I_N(\mathcal{\bar
R}_N),G_{i-1}^{0,1}\bigg)_{0,0}}{\|G_i^{0,-1}\|_{0,-1}^2}
G_i^{0,-1}(v)}.
\end{equation}

Using the relations (\ref{18}) and (\ref{cc6}) yields $\mathcal
I_N(\mathcal{\bar R}_N)=0$. Thus
\[
\mathcal I_N\bigg(\mathcal M^q \bar u_N-\bar u_N(v)-\lambda \mathcal
\mathcal K_{N,\theta}(\bar u_N)-\bar f\bigg)=0,
\]
which can be rewritten as
\begin{equation}\label{rv9c}
\bar u_N(v)=\mathcal I_N\bigg(\mathcal M^q -\lambda \mathcal
\mathcal K_{N,\theta}\bigg)\bar u_N-\mathcal I_N \bar f=\mathcal I_N
 \mathcal T_N \bar u_N -\mathcal I_N \bar f,
\end{equation}
and thereby
\begin{equation}\label{rv9}
\bar u_N(v)=\Pi_N \mathcal T \bar u_N+\mathcal S_N \bar u_N
-\mathcal I_N \bar f,
\end{equation}
where $\mathcal I_N$ and $\Pi_N$ are defined in (\ref{9}) and
(\ref{cc5}) respectively and
\begin{eqnarray*}
\mathcal T&=&\mathcal M^q-\lambda \mathcal K_\theta,\\
\mathcal T_N&=&\mathcal M^q -\lambda \mathcal \mathcal K_{N,\theta},\\
\mathcal S_N&=&\mathcal I_N \mathcal T_N-\Pi_N \mathcal T.
\end{eqnarray*}

Since (\ref{rv9}) is the form in which the discrete Galerkin method
is implemented, as it leads directly to the equivalent linear system
(\ref{17}). It can be easily check that (\ref{rv9}) can be
considered in the operator form (\ref{rv8}) by assuming
\begin{eqnarray}\label{cc1}
\nonumber\mathcal X&=&H_{0,0}^1(\Omega),~~\mathcal
Y=L^2(\Omega),~~\tilde{\mathcal
Y}=\mathcal P_N,\\
v_N&=&\bar u_N,~~\quad ~~\tilde{\mathcal B_N}=\Pi_N,~~\quad \mathcal
B_N=\mathcal
I_N,\\
\nonumber\mathcal A&=&\mathcal T,~~\quad ~~ \mathcal A_N=\mathcal
I_N \mathcal T_N,
\end{eqnarray}
which can be completed the desired result of step 1.

{\bf Step 2:} In this step we intend to apply Theorem \ref{rt1} with
the assumptions (\ref{cc1}) to prove the Theorem. To this end,
following Theorem \ref{rt1} we must show that
\begin{eqnarray}\label{cc3}
\nonumber&&1)~\text{for any}~ Z \in \mathcal P_N ~\text{we have}~
\|\mathcal S_N Z\|_{0,0}=\|\mathcal I_N \mathcal T_N Z-\Pi_N
\mathcal T Z\|_{0,0}
\to 0 ~\text{as}~ N \to \infty,\\
&&2) \|\mathcal T-\Pi_N \mathcal T\|_{0,0} \to 0~ \text{as}~ N \to
\infty,\\
\nonumber&&3) \|\bar f-\mathcal I_N \bar f\|_{0,0} \to 0 ~\text{as}~
N \to \infty.
\end{eqnarray}

First, we prove the first condition in (\ref{cc3}). For this, we can
write
\begin{equation}\label{rv13}
\|\mathcal S_N Z\|_{0,0}=\|\bigg(\mathcal I_N \mathcal T_N-\Pi_N
\mathcal T\bigg)Z\|_{0,0}\le \|\mathcal I_N\bigg(\mathcal
T_N-\mathcal T\bigg)Z\|_{0,0}+\|\bigg(\mathcal
I_N-\Pi_N\bigg)\mathcal T Z\|_{0,0}.
\end{equation}

Since $\mathcal M^q Z \in \mathcal P_N$ for any $Z \in \mathcal P_N$
then
\begin{equation}\label{rv14}
\|\mathcal S_N Z\|_{0,0}\le \lambda\Bigg(\|\mathcal
I_N\bigg(\mathcal K_{N,\theta}-\mathcal
K_{\theta}\bigg)Z\|_{0,0}+\|\bigg(\mathcal I_N-\Pi_N\bigg)\mathcal
K_\theta Z\|_{0,0}\Bigg).
\end{equation}

Using Lemmas \ref{l3} and \ref{l4} and the relations (\ref{rv5}) and
(\ref{rv6}) we can obtain
\begin{equation}\label{25rv}
\|\mathcal I_N\bigg(\mathcal K_{N,\theta}-\mathcal
K_{\theta}\bigg)z\|_{0,0} \le \sup\limits_{v \in \Omega}{|\mathcal
K_\theta(z(v))-\mathcal K_{N,\theta}(z(v))|} \le C_1 N^{-k_1}
\sup\limits_{v \in \Omega}{\bigg(\|\tilde
K(v,w(\theta))\|_{0,0,k_1}\|v Z(w(\theta))\|_{0,0}\bigg)},
\end{equation}
where norms $\|\tilde K(v,w(\theta))\|_{0,0,k_1}$ and $\|v
Z(w(\theta))\|_{0,0}$ are applied with respect to the variable
$\theta$.

According to Lemma \ref{l4} we have
\begin{equation}\label{cc7}
\Bigg|\bigg(\mathcal K_\theta,J_s^{0,0}\bigg)_{0,0}-\bigg(\mathcal
K_\theta,J_s^{0,0}\bigg)_{N,0,0}\Bigg| \le CN^{-k_2} \|\mathcal
K_\theta\|_{0,0,k_2} \|J_s^{0,0}\|_{0,0},~~s=0,1,...,N,
\end{equation}
where norm $\|\mathcal K_\theta\|_{0,0,k_2}$ is applied with respect
to the variable $v$. Using (\ref{cc7}) and the relations (\ref{9}),
(\ref{cc5}) we can yields
\begin{equation}\label{rv15}
\|\bigg(\mathcal I_N-\Pi_N\bigg)\mathcal K_\theta Z\|_{0,0} \to
0~~\text{as}~~N \to \infty.
\end{equation}

Substituting (\ref{25rv}) and (\ref{rv15}) in (\ref{rv14}) we can
conclude the first condition in (\ref{cc3}). Applying Lemma \ref{l4}
gives us the third condition in (\ref{cc3}). To complete the proof,
it is sufficient that we prove the second condition of (\ref{cc3}).
To this end, we apply Lemma \ref{rl1} with the assumptions
(\ref{cc1}). Since $\|y-\Pi_N y\|_{0,0} \to 0~\text{as}~N \to
\infty$ for $y \in L^2(\Omega)$(see[]), the second condition in
(\ref{cc3}) can be achieved by proving compactness of the operator
$\mathcal T$. Since $\tilde k(v,w)$ is a continuous kernel, then the
operator $\mathcal T$ will be compact if $\mathcal M^q(\bar
u):H_{0,0}^1(\Omega) \to L^2(\Omega)$ be a compact operator. For
this, define
$M=\big[\int_{0}^{1}{\int_{0}^{1}{\bigg(v^{\frac{1}{q}}-w^{\frac{1}{q}}\bigg)^{-2q}
dw}dv}\big]^{\frac{1}{2}}< \infty$. For $\bar u \in
H_{0,0}^1(\Omega)$, using Cauchy-Schwartz inequality we have
\begin{eqnarray*}
\|\mathcal M^q \bar u\|_{0,0}^2&=&
\int_{0}^{1}{|\int_{0}^{v}{\bigg(v^{\frac{1}{q}}-w^{\frac{1}{q}}\bigg)^{-q}
\bar u'(w) dw}|^2 dv} \le
\int_{0}^{1}\bigg({\int_{0}^{v}{|\bigg(v^{\frac{1}{q}}-w^{\frac{1}{q}}\bigg)^{-q} \bar u'(w)| dw}\bigg)^2 dv} \\
& \le &
\int_{0}^{1}\bigg({\int_{0}^{1}{|\bigg(v^{\frac{1}{q}}-w^{\frac{1}{q}}\bigg)^{-q}
\bar u'(w)| dw}\bigg)^2 dv} \le
\int_{0}^{1}{\bigg(\int_{0}^{1}{\bigg(v^{\frac{1}{q}}-w^{\frac{1}{q}}\bigg)^{-2q}dw}
\int_{0}^{1}{|\bar u'(w)|^2dw}\bigg)dv} \\
& \le & M^2 \|\bar u'\|_{0,0}^2 \le M^2 \|\bar
u\|_{H_{0,0}^1(\Omega)}^2.
\end{eqnarray*}

So $\mathcal M^q \bar u \in L^2(\Omega)$ and $\|\mathcal M^q\|=\sup
\bigg\{\frac{\|\mathcal M^q \bar u\|_{0,0}}{\|\bar
u\|_{H_{0,0}^1(\Omega)}}~~\Big|~~ \bar u \ne 0, \bar u \in
H_{0,0}^1(\Omega)\bigg\} \le M < \infty$. Therefore $\mathcal
M^q:H_{0,0}^{1}(\Omega) \to L^2(\Omega)$ defined by (\ref{rv4}), is
a bounded operator. If we proceed same as Theorem 3.4 in
[\ref{arr1}], we can conclude compactness of the operator $\mathcal
M^q(\bar u)$. Then the operator $\mathcal T$ is a compact operator
and from Lemma \ref{rl1} we have
\[
\|\mathcal T-\Pi_N \mathcal T\|_{0,0} \to 0~\text{as}~ N\to \infty
\]

In this position, all conditions in (\ref{cc3}) that are required to
deduce existence and uniqueness of solutions of the discrete
Galerkin system (\ref{17}) have been proved and then the proof is
completed.
\end{proof}
\section{Convergence analysis}
In this section, we will try to provide a reliable convergence
analysis which theoretically justify convergence of the proposed
discrete Galerkin method for the numerical solution of a special
case of (\ref{6x}) with $\bar p(v)=1$.

In the sequel, our discussion is based on these Lemmas:
\begin{lemma}\cite{aa26}\label{l2}For any $Z \in B_{0,0}^1(\Omega)$, we have
\[
\|\mathcal I_{N}Z\|_{0,0} \le
C\bigg(\|Z\|_{0,0}+N^{-1}\|Z'\|_{1,1}\bigg).
\]
\end{lemma}
\begin{lemma}\label{l5} The fractional integral operator $\mathcal
I^\mu$ is bounded in $L^2(\Omega)$ and
\[
\|\mathcal I^\mu Z\|_{0,0} \le C \|Z\|_{0,0},~~ Z \in L^2(\Omega).
\]
\end{lemma}
\begin{lemma}\label{l6} \cite{r6}(Gronwall inequality)Assume that $Z(v)$ is a
non-negative, locally integrable function defined on $\Omega$ which
satisfies
\[
Z(v) \le b(v)+B \int_{0}^{v}{(v-w)^\alpha w^\beta Z(w) dw}, \quad w
\in \Omega,
\]
where $\alpha, \beta>-1$,~$b(v) \ge 0$ and $B \ge 0$. Then, there
exist a constant $C$ such that
\[
Z(v) \le b(v)+ C \int_{0}^{v}{(v-w)^\alpha w^\beta b(w) dw}, \quad w
\in \Omega.
\]
\end{lemma}
Now, we state and prove the main result of this section regarding
the error analysis of the proposed method for the numerical solution
of a special case of (\ref{6x}) with $\bar p(v)=1$.
\begin{theorem}\label{t1}
(Convergence)Let $u(x)$ and $\bar u(v)$ are the exact solutions of
the equations (\ref{1}) and (\ref{6x}) respectively that is related
by $u(x)=\bar u(v)$. Assume that $u_N(x)=\bar u_N(v)$ be the
approximate solution of (\ref{1}), where $\bar u_N(v)$ is the
discrete Galerkin solution of the transformed equation (\ref{6x}).
If the following conditions are fulfilled
\begin{enumerate}
  \item $\bar f(v) \in B_{0,0}^{k_1}(\Omega)$ for $k_1 \ge 1$,
  \item $\mathcal K(\bar u) \in B_{0,0}^{k_2}(\Omega)$ for $k_2 \ge 1$,
\end{enumerate}
then for sufficiently large $N$ we have
\[
\|e_N(u)\|_{0,0} \le C N^{-\xi}\bigg(|\bar f|_{0,0,k_1}+|\mathcal
K(\bar u)|_{0,0,k_2}\bigg)
\]
where $\xi=\min\{k_1,k_2\}$ and $e_N(u)= u(x)-u_N(x)$ denotes the
error function.
\end{theorem}
\begin{proof}

Since $\mathcal M^q \bar u_N$ is a polynomial from degree of at most
$N$ then we can rewrite (\ref{rv9c}) as
\begin{equation}\label{7x}
\mathcal M^q \bar u_N-\bar u_N(v)-\mathcal I_N\bigg(\lambda \mathcal
\mathcal K_{N,\theta}(\bar u_N)\bigg)=\mathcal I_N \bar f.
\end{equation}

Subtracting (\ref{6x}) from (\ref{7x}) and some simple manipulations
we can obtain
\begin{equation}\label{19}
\mathcal M^q \bar e_N=\bar e_N(v)+e_{\mathcal I_N}(\bar f)+\lambda
\mathcal K(\bar e_N)+\lambda e_{\mathcal I_N}(\mathcal K(\bar
u_N))+\lambda \mathcal I_N\bigg(\mathcal K_{\theta}(\bar
u_N)-\mathcal K_{N,\theta}(\bar u_N)\bigg),
\end{equation}
where $e_{\mathcal I_N}(g)=g-\mathcal I_N(g)$ denotes the
interpolating error function and $\bar e_N(v)=\bar u(v)-\bar u_N(v)$
is the error function of approximating (\ref{6x}) using discrete
Galerkin solution $\bar u_N(v)$.

Applying the transformation (\ref{6xx}) to the operator $\mathcal
I^q u$ we get
\[
\mathcal M_1^q \bar
u=\frac{1}{\Gamma{(q)}}\int\limits_0^v{\bigg(v^{\frac{1}{q}}-w^{\frac{1}{q}}\bigg)^{q-1}
\bar u(w) \frac{w^{\frac{1}{q}-1}}{q}dw}.
\]

Following (\ref{20}) it is easy to check that
\begin{equation}\label{7xx}
\mathcal M_1^q \big(\mathcal M^q \bar u(v)\big)=\bar u(v)-\bar u(0).
\end{equation}

Applying the operator $\mathcal M_1^q$ on the both sides of
(\ref{19}) and using (\ref{7xx}) we can yield
\begin{equation}\label{rv1}
|\bar e_N(v)| \le  \mathcal M_1^q(|\bar e_N|+|\lambda \mathcal
K(\bar e_N)|)+\bigg|\mathcal M_1^q\bigg(e_{\mathcal I_N}(\bar
f)+\lambda e_{\mathcal I_N}(\mathcal K(\bar u_N))+\lambda \mathcal
I_N\bigg(\mathcal K_{\theta}(\bar u_N)-\mathcal K_{N,\theta}(\bar
u_N)\bigg)\bigg)\bigg|.
\end{equation}

Since
\begin{eqnarray}\label{rv2}
\nonumber\mathcal M_1^q \int\limits_{0}^{v}{|\tilde K(v,w) \bar
e_N(w)|
dw}&=&\frac{1}{\Gamma{(q)}}\int\limits_{0}^{v}{\int\limits_{0}^{w}{\big(v^{\frac{1}{q}}-w^{\frac{1}{q}}\big)^{q-1}
\frac{w^{\frac{1}{q}-1}}{q} |\tilde K(w,s)||\bar e_N(s)|ds}dw}\\
\nonumber&=&\frac{1}{\Gamma{(q)}}\int\limits_{0}^{v}{\int\limits_{s}^{v}{\big(v^{\frac{1}{q}}-w^{\frac{1}{q}}\big)^{q-1}
\frac{w^{\frac{1}{q}-1}}{q} |\tilde K(w,s)||\bar e_N(s)|dw}ds}\\
&=&\frac{1}{\Gamma{(q)}}\int\limits_{0}^{v}{\tilde K_1(v,s) |\bar
e_N(s)|ds}\le \frac{\|\tilde
K_1(v,s)\|_\infty}{\Gamma{(q)}}\int\limits_{0}^{v}{|\bar e_N(s)|ds},
\end{eqnarray}
then (\ref{rv1}) can be rewritten as
\begin{eqnarray}\label{rv3}
\nonumber|\bar e_N(v)| &\le &
\frac{1}{\Gamma{(q)}}\bigg(\int\limits_{0}^{v}{\big(v^{\frac{1}{q}}-w^{\frac{1}{q}}\big)^{q-1}
\frac{w^{\frac{1}{q}-1}}{q}|\bar e_N(w)|dw}+|\lambda|\|\tilde
K_1(v,s)\|_\infty \int\limits_{0}^{v}{|\bar e_N(w)|
dw}\bigg)\\
\nonumber &+&\bigg|\mathcal M_1^q\bigg(e_{\mathcal I_N}(\bar
f)+\lambda e_{\mathcal I_N}(\mathcal K(\bar u_N))+\lambda \mathcal
I_N\bigg(\mathcal K_{\theta}(\bar u_N)-\mathcal K_{N,\theta}(\bar
u_N)\bigg)\bigg)\bigg|\\ &\le&  \tilde C
\bigg(\int\limits_{0}^{v}{\big(v^{\frac{1}{q}}-w^{\frac{1}{q}}\big)^{q-1}
|\bar e_N(w)|dw}\bigg)\\ \nonumber &+&\bigg|\mathcal
M_1^q\bigg(e_{\mathcal I_N}(\bar f) +\lambda e_{\mathcal
I_N}(\mathcal K(\bar u_N))+\lambda \mathcal I_N\bigg(\mathcal
K_{\theta}(\bar u_N)-\mathcal K_{N,\theta}(\bar
u_N)\bigg)\bigg)\bigg|,
\end{eqnarray}
where
\[\tilde K_1(v,s)=\int\limits_{s}^{v}{\big(v^{\frac{1}{q}}-w^{\frac{1}{q}}\big)^{q-1}
\frac{w^{\frac{1}{q}-1}}{q} |\tilde K(w,s)|dw},\quad \tilde
C=\max\bigg\{\frac{1}{q \Gamma{(q)}},\frac{|\lambda|\|\tilde
K_1(v,s)\|_\infty }{\Gamma{(q)}}\bigg\}
\|1+\big(v^{\frac{1}{q}}-w^{\frac{1}{q}}\big)^{1-q}\|_\infty.\]

Using Gronwall inequality(Lemma \ref{l6}) in (\ref{rv3}) yields
\begin{equation}\label{21c}
\|\bar e_N\|_{0,\frac{1}{q}-1} \le C_1 \bigg(\|\mathcal
M_1^q\bigg(e_{\mathcal I_N}(\bar f)+\lambda e_{\mathcal
I_N}(\mathcal K(\bar u_N))+\lambda \mathcal I_N\bigg(\mathcal
K_{\theta}(\bar u_N)-\mathcal K_{N,\theta}(\bar
u_N)\bigg)\bigg)\|_{0,\frac{1}{q}-1}\bigg).
\end{equation}

It can be checked that by applying the transformation (\ref{6}) we
obtain \begin{equation}\label{cc8}\|\bar
e_N(v)\|_{0,\frac{1}{q}-1}^2=q
2^{\frac{1}{2}-1}\|e_N(x)\|_{0,0}^2.\end{equation}

On the other hand, by applying (\ref{6}) and Lemma \ref{l5}  we can
get
\begin{equation}\label{cc4}
\|\mathcal M_1^q \bar Z(v)\|_{0,\frac{1}{q}-1}^2=q
2^{\frac{1}{q}-1}\|\mathcal I^q Z(x)\|_{0,0}^2 \le C q
2^{\frac{1}{q}-1}\|Z(x)\|_{0,0}^2=C \|\bar
Z(v)\|_{0,\frac{1}{q}-1}^2,
\end{equation}
where $Z(x)$ is a given function and $\bar Z(v)=z(v^{\frac{1}{q}})$.
Using the relations (\ref{cc8}) and (\ref{cc4}) in (\ref{21c}) we
can obtain
\begin{eqnarray}\label{21cc}
\sqrt{q 2^{\frac{1}{q}-1}}\|e_N\|_{0,0} &\le & C_2 \| e_{\mathcal
I_N}(\bar f)+\lambda e_{\mathcal I_N}(\mathcal K(\bar u_N))+\lambda
\mathcal I_N\bigg(\mathcal K_{\theta}(\bar u_N)-\mathcal
K_{N,\theta}(\bar u_N)\bigg)\|_{0,\frac{1}{q}-1} \\
\nonumber&\le & C_2 \bigg(\| e_{\mathcal I_N}(\bar
f)\|_{0,\frac{1}{q}-1}+|\lambda|\|e_{\mathcal I_N}(\mathcal K(\bar
u_N))\|_{0,\frac{1}{q}-1}+|\lambda|\|\mathcal I_N\bigg(\mathcal
K_{\theta}(\bar u_N)-\mathcal K_{N,\theta}(\bar
u_N)\bigg)\|_{0,\frac{1}{q}-1}\bigg).
\end{eqnarray}

Since $\delta^{0,\frac{1}{q}-1} \le \delta^{0,0}$ then we have
$\|.\|_{0,\frac{1}{q}-1} \le \|.\|_{0,0}$ and thereby (\ref{21cc})
can be written as
\begin{equation}\label{21}
\|e_N\|_{0,0} \le C_3 \bigg(\mathcal L_1+\mathcal L_2+\mathcal
L_3\bigg),
\end{equation}
where
\[
\mathcal L_1=\|e_{\mathcal I_N}(\bar f)\|_{0,0}, \quad \mathcal
L_2=\|e_{\mathcal I_N}(\mathcal K(\bar u_N))\|_{0,0}, \quad \mathcal
L_3=\|\mathcal I_N\big(\mathcal K_{\theta}(\bar u_N)-\mathcal
K_{N,\theta}(\bar u_N)\big)\|_{0,0}. \]

Now, it is sufficient to find suitable upper bounds to the above
norms. According to Lemma \ref{l1} we have
\begin{eqnarray}\label{22}
\mathcal L_1&=&\|e_{\mathcal I_N}(\bar f)\|_{0,0} \le C_3 N^{-k_1}
|\bar f|_{0,0,k_1},\\
\nonumber \\
\nonumber\mathcal L_2&=&\|e_{\mathcal I_N}(\mathcal K(\bar
u_N))\|_{0,0} \le C_4 N^{-k_2}|\mathcal K(\bar u-\bar
e_N(u))|_{0,0,k_2}.
\end{eqnarray}

Using Lemmas \ref{l3} and \ref{l4} we can obtain
\begin{eqnarray}\label{25}
\nonumber \mathcal L_3=\|\mathcal I_N\big(\mathcal K_\theta(\bar
u_N)-\mathcal K_{N,\theta}(\bar u_N)\big)\|_{0,0} &\le &
\sup\limits_{v \in
\Omega}{|\mathcal K_\theta(\bar u_N)-\mathcal K_{N,\theta}(\bar u_N)|}\\
\nonumber &\le & C_7 N^{-k_3} \sup\limits_{v \in
\Omega}{\bigg(\|\tilde
K(v,w(\theta))\|_{0,0,k_3}\|v \bar u_N(w(\theta))\|_{0,0}\bigg)}\\
\nonumber &\le & C_7 N^{-k_3}\bigg(\sup\limits_{v \in
\Omega}{\|\tilde
K(v,w(\theta))\|_{0,0,k_3}\bigg)}\|\bar u_N(w)\|_{0,0}\\
&\le & C_7 N^{-k_3} \bigg(\sup\limits_{v \in \Omega}{\|\tilde
K(v,w(\theta))\|_{0,0,k_3}\bigg)}\bigg(\|\bar u\|_{0,0}+\|\bar
e_N\|_{0,0}\bigg),
\end{eqnarray}
where norms $\|\tilde K(v,w(\theta))\|_{0,0,k_3}$ and $\|v \bar
u_N(w(\theta))\|_{0,0}$ are applied with respect to the variable
$\theta$.

For sufficiently large $N$ the desired result can be concluded by
inserting (\ref{22}) and (\ref{25}) in (\ref{21}).
\end{proof}

\section{Numerical Results}
In this section we apply a program written in Mathematica to a
numerical example to demonstrate the accuracy of the proposed method
and effectiveness of applying GJPs. the "Numerical Error" always
refers to the $L^2$-norm of the obtained error function.
\begin{example}\label{e1}Consider the following FIDE
\[\mathcal D^{\frac{1}{2}}u(x)=u(x)+f(x)+\frac{1}{2}\int\limits_0^x{\sqrt{x t} u(t)
dt},\quad u(0)=0,\] with the exact solution
$u(x)=\frac{\sin{x}}{\sqrt{x}}$ and
\[
f(x) =\frac{-x+\sqrt{x}\bigg(\sqrt{x}
\cos{x}+\sqrt{\pi}\bigg(J_0(\frac{x}{2})\cos{\frac{x}{2}}-J_1(\frac{x}{2})\sin{\frac{x}{2}}\bigg)\bigg)-2
\sin{x}}{2 \sqrt{x}}.
\]

Here $J_n(z)$ gives the Bessel function of the first kind.
\end{example}
This example has a singularity at the origin, i.e.,
\[
u'(x) \sim \frac{1}{\sqrt{x}}.
\]

In the theory presented in the previous section, our main concern is
the regularity of the transformed solution. For the present problem
by applying coordinate transformation $x=v^2$, the infinitely smooth
solution
\[
\bar u(v)=u(v^2)=\frac{\sin{v^2}}{v},
\]
is obtained. The main purpose is to check the convergence behavior
of the numerical solutions with respect to the approximation degree
$N$. Numerical results obtained are given in the Table 1 and Figure
1. As expected, the errors show an exponential decay, since in this
semi-log representation(Figure 1) one observes that the error
variations are essentially linear versus the degrees of
approximation.
\begin{center} \centerline{\footnotesize Table $1$:
The numerical errors of example \ref{e1}. } \vspace{2mm}
\begin{tabular}{ cccc }
\hline
N \hspace{1. cm}  & \hspace{-.7 cm} {\small  Numerical Errors}\\
\hline\hline
2\hspace{1. cm} &\hspace{-.7 cm}$7.86 \times 10^{-3}$\\
4\hspace{1. cm} &\hspace{-.7 cm}$4.71 \times 10^{-5}$\\
6\hspace{1. cm} &\hspace{-.7 cm}$2.15 \times 10^{-6}$\\
8\hspace{1. cm} &\hspace{-.7 cm}$1.05 \times 10^{-8}$\\
10\hspace{1. cm} &\hspace{-.7 cm}$4.42 \times 10^{-10}$\\
12\hspace{1. cm} &\hspace{-.7 cm}$2.05 \times 10^{-12}$\\
14\hspace{1. cm} &\hspace{-.7 cm}$2.64 \times 10^{-14}$\\
16\hspace{1. cm} &\hspace{-.7 cm}$1.91 \times 10^{-16}$\\
\hline\\
\end{tabular}
\end{center}

\begin{figure}[!h]
\begin{center}
\input{epsf}
\epsfxsize=2.8 in \epsfysize=2.8 in \epsffile{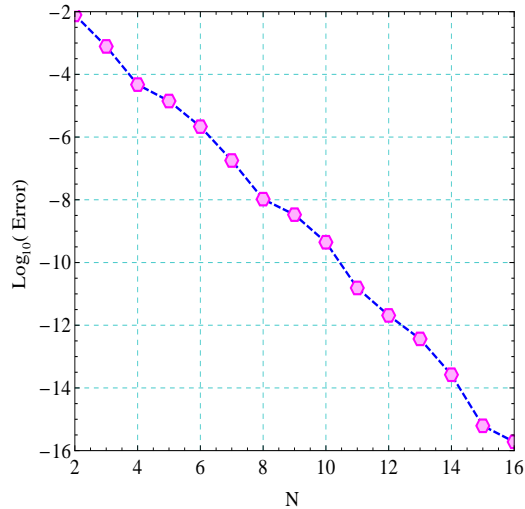}
\caption{The numerical errors of example \ref{e1} with various
values of N.}
\end{center}
\end{figure}


\newpage
\begin{thebibliography}{9}
\bibitem{rvv4}\label{rvv4} K. E. Atkinson, {\it The Numerical Solution of Integral Equations
of the Second Kind,} Cambridge, (1997).

\bibitem{r1}\label{r1}F. Awawdeh, E. A. Rawashdeh, H. M. Jaradat,
Analytic solution of fractional integro-differential equations, {\it
Ann. Univ. Craiova math. Comput. Sci. Ser.} {\bf 38} (2011), 1-10.

\bibitem{r2}R. L. Bagley, P. J. Torvik, A theoretical basis for
the application of fractional calculus to viscoelasticity, {\it J.
Rheol.,} {\bf 27}(1983), 201-210.

\bibitem{r3}\label{r3}H. Brunner, {\it Collocation Methods for Volterra and
Related Functional Equations,} Cambridge University Press,
Cambridge, (2004).

\bibitem{r5}\label{r5} M. Caputo, Linear models of dissipation whose Q is
almost frequency independent II, {\it Geophys. J. Roy. Astron.
Soc.,} {\bf 13}(1967), 529-539.

\bibitem{a13}
\label{a13} C. Canuto, M. Y. Hussaini, A. Quarteroni, T. A. Zang,
{\it Spectral Methods, Fundamentals in Single Domains,}
Springer-Verlag, Berlin, (2006).

\bibitem{r6}\label{r6}Y. Chen~Y., T. Tang, Convergence analysis of the Jacobi
spectral collocation methods for volterra integral equations with a
weakly singular kernel, {\it Math. Comput.} {\bf 79}(2010), 147-167.

\bibitem{r8}\label{r8}K. Diethelm, {\it The Analysis of Fractional Differential
Equations,} Springer-Verlag Berlin, (2010).

\bibitem{aa16}
\label{aa16} E. H. Doha, A. H. Bhrawy, S. S. Ezz-Eldien, A new
Jacobi operational matrix: an application for solving fractional
differential equations, {\it Appl. Math. Model.,} {\bf 36} (2012),
4931-4943.

\bibitem{r9}\label{r9}M. R. Eslahchi, M. Dehghan, M. Parvizi, Application
of the Collocation method for solving nonlinear fractional
integro-differential equations, {\it J. Comput. Appl. Math.,} {\bf
257} (2014), 105-128.

\bibitem{ax15}
\label{ax15} F. Ghoreishi, P. Mokhtary, Spectreal Collocation method
for multi-order fractional differential equations, {\it Int. J.
Comput. Math.,} {\bf 11}(2014), no. 5, 23pp.

\bibitem{a15}
\label{a15} Ben-Yu Guo , Jie Shen, Li-Lian Wang, Optimal
Spectral-Galerkin methods using Generalized Jacobi polynomials, {\it
J. Sci. Comput.,} {\bf 27}(2006), 305-322.

\bibitem{a16}
\label{a16} Ben-Yu Guo, Jie Shen, Li-Lian Wang, Generalized Jacobi
polynomials/functions and their applications, {\it Appl. Numer.
Math.,} {\bf 59}(2009), 1011-1028.

\bibitem{r11}\label{r11} L. Huang, X. F. Li, Y. L. Zhao, X. Y. Duan,
Approximate solution of fractional integro-differential eqations by
Taylor expansion method, {\it Comput. Math. Appl.,} {\bf 62} (2011),
1127-1134.

\bibitem{a14}
\label{a14} J. S. Hesthaven, S. Gottlieb, D. Gottlieb, {\it Spectral
Methods for Time-Dependent Problems,} Cambridge University Press,
(2007).

\bibitem{rvv1}\label{rvv1} L. V. Kantrovich,
Functional analysis and applied mathematics,{\it Usp. Mat. Nauk,}
{\bf 3} (1984), 89-185.

\bibitem{rvv2}\label{rvv2} L. V. Kantrovich, G. P. Akilov, {\it Functional Analysis in Normed Spaces(Funktsional'nyi
analiz v normirovannykh prostranstvakh),} Fizmatgiz, Moscow, (1959).

\bibitem{r11}\label{r11} L. Huang, X. F. Li, Y. L. Zhao, X. Y. Duan,
Approximate solution of fractional integro-differential eqations by
Taylor expansion method, {\it Comput. Math. Appl.,} {\bf 62} (2011),
1127-1134.

\bibitem{r13}\label{r13}A. A. Kilbas, H. M. Srivastava, J. J. Trujillo,
{\it Theory and Applications of Fractional Differential Equations,}
Elsevier, Amesterdam, (2006).

\bibitem{r12}\label{r12}M. M. Khader, N. H. Sweilam, On the approximate
solutions for system of fractional integro-differential equations
using Chebyshev pseudo-spectral method, {\it Appl. Math. Model.,}
{\bf 27}(2013), no. 24, 819-9828.

\bibitem{r17}\label{r17}R. C. Mittal, R. Nigam, Solution of fractional
calculus and fractional integro-differential equations by Adomian
decomposition method, {\it Int. J. Appli. Math. Mech.,} {\bf
4}(2008), no. 4, 87-94.

\bibitem{r16}\label{r16}Xiaohua Ma, C. Huang, Numerical solution of
fractional integro-differential equations by Hybrid Collocation
method, {\it Appl. Math. Comput.,}  {\bf 219} (2013), 6750-6760.

\bibitem{r15}\label{r15}Xiaohua Ma, C. Huang, Spectral Collocation method
for linear fractional integro-differential equations, {\it Appl.
Math. Model.} {\bf 38} (2014), no. 4, 1434-1448.

\bibitem{r18}\label{r18}P. Mokhtary, F. Ghoreishi, The $L^2$-convergence of
the Legendre spectral Tau matrix formulation for nonlinear
fractional integro differential equations, {\it Numer. Algorithms},
{\bf 58} (2011), 475-496.

\bibitem{r19}\label{r19}D. Nazari, S. Shahmorad, Application of the
fractionl differential transform method to fractional-order
integro-differential equations with nonlocal boundary conditions,
{\it J. Comput. Appl. Math.,} {\bf 234} (2010), 883-891.

\bibitem{r20}\label{r20}K. B. Oldham, J. Spanier, {\it The Fractional Calculus,}
Academic Press, New York, (1974).

\bibitem{r21}\label{r21}W. E. Olmstead, R. A. Handelsman, Diffusion in a
semi-infinite region with nonlinear surface dissipation, {\it SIAM
Rev.} {\bf 18} (1976), 275-291.

\bibitem{r22}\label{r22}I. Podlubny, {\it Fractional Differential Equations,}
Academic Press, (1999).

\bibitem{arr1}
\label{arr1} D. Porter,  David S. G. Stirling, {\it Integral
Equations, A Practical Treatment, from Spectral Theory to
Applications}, Cambridge University Press, New York, (1990).


\bibitem{r23}\label{r23}E. A. Rawashdeh, Nemerical solution of fractional
integro-differential equations by Collocation method, {\it Appl.
Math. Comput,} {\bf 176}(2006), 1-6.

\bibitem{aa26}\label{aa26}
Jie Shen, Tao Tang, Li-Lian Wang,{\it Spectral Methods, Algorithms,
Analysis and Applications,} Springer, (2011).

\bibitem{r7}\label{r7}
Li Xianjuan, T. Tang, Convergence analysis of Jacobi spectral
Collocation methods for Abel-Volterra integral equations of second
kind, {\it Front. Math. China,} {\bf 7}(2012), no. 1, 69-84.

\bibitem{r25}\label{r25}Li Zhu, Qibin Fan, Numerical solution of nonlinear
fractional order Volterra integro differential equations by SCW,
{\it Commun. Nonlinear Sci. Numer. Simul.,} {\bf 18} (2013), no. 15,
1203-1213.

\bibitem{rvv3}\label{rvv3}G. M. Vainikko, Galerkin's perturbation method and the general theory of approximate
methods for non-linear equations,{\it USSR Computational Mathematics
and Mathematical Physics  ,} {\bf 7}(1967), no. 4, 1-41.
\end{thebibliography}
\end{document}